\documentclass[12pt]{paper}
\usepackage{amssymb, amsmath, amsthm, booktabs, float, framed}
\usepackage{lscape}

\usepackage{booktabs}
\usepackage{url, hyperref}
\usepackage{anysize}
\usepackage[shortlabels]{enumitem}
\usepackage{tikz}
\usepackage{dsfont}

\usepackage{listings}
\usepackage{mathrsfs}

\newcommand{\allones}{\mathds{1}}
\newcommand{\DW}{\mathsf{DW}}
\renewcommand{\DH}{\mathsf{DH}}
\newcommand{\DQ}{\mathsf{DQ}}
\newcommand{\pg}{\mathsf{pg}}
\newcommand{\Q}{\mathsf{Q}}
\newcommand{\PG}{\mathsf{PG}}
\newcommand{\PSL}{\mathsf{PSL}}
\makeatletter
\newcommand{\bigperp}{%
  \mathop{\mathpalette\bigp@rp\relax}%
  \displaylimits
}

\newcommand{\bigp@rp}[2]{%
  \vcenter{
    \m@th\hbox{\scalebox{\ifx#1\displaystyle2.1\else1.5\fi}{$#1\perp$}}
  }%
}
\makeatother

 \newtheorem{theorem}{Theorem}[section]
 \newtheorem{lemma}[theorem]{Lemma}
 \newtheorem{corollary}[theorem]{Corollary}
 
\newtheorem{conjecture}[theorem]{Conjecture}
\newtheorem{xmpl}[theorem]{Example}
\newtheorem{problem}[theorem]{Problem}

\newtheorem{remark}[theorem]{Remark}

\renewcommand\le{\leqslant}
\renewcommand\ge{\geqslant}

\newcounter{followUpMarkerA}
\newcounter{followUpMarkerB}
\def\followups{ }
\newcommand{\followupAdd}[1]{\expandafter\def\expandafter\followups\expandafter{\followups{}\item \stepcounter{followUpMarkerB} Page \pageref{followupmarker\thefollowUpMarkerB}: #1}}

\marginsize{2cm}{2cm}{2cm}{2.5cm}
\hypersetup{citecolor=blue, linkcolor=blue, colorlinks=true}

\title{Implications of vanishing Krein parameters on\\ Delsarte designs, with applications in finite geometry}
\shorttitle{Implications of vanishing Krein parameters on Delsarte designs}
\author{John Bamberg, Jesse Lansdown\thanks{This paper builds on part of the doctoral work of the second author under the supervision of the first.
The second author acknowledges the support of an Australian Government Research Training Program Scholarship and Australian Research Council Discovery Project DP200101951. We would also like to thank the anonymous referees and the editor for their valuable comments.}}

\institution{Centre for the Mathematics of Symmetry and Computation\\
Department of Mathematics and Statistics\\
The University of Western Australia, W.A., Australia\\
\texttt{John.Bamberg@uwa.edu.au, Jesse.Lansdown@uwa.edu.au}
}

\allowdisplaybreaks

\begin{document}
\maketitle

\begin{abstract}
In this paper we show that  if $\theta$ is a $T$-design of an association scheme $(\Omega, \mathcal{R})$, and the Krein parameters $q_{i,j}^h$ vanish for some $h \not \in T$ and all $i, j \not \in T$  ($i, j, h \neq 0$), then $\theta$ consists of precisely half of the vertices of $(\Omega, \mathcal{R})$ or it is a $T'$-design, where $|T'|>|T|$. 
We then apply this result to various problems in finite geometry. In particular, we show for the first time that nontrivial $m$-ovoids of generalised octagons of order $(s, s^2)$ do not exist. We give short proofs of similar results for
(i) partial geometries with certain order conditions;
(ii) thick generalised quadrangles of order $(s,s^2)$;
(iii) the dual polar spaces
$\DQ(2d, q)$, $\DW(2d-1,q)$ and $\DH(2d-1,q^2)$, for $d \ge 3$;
(iv) the Penttila-Williford scheme.
In the process of (iv), we also consider a natural generalisation of the Penttila-Williford scheme in $\Q^-(2n-1, q)$, $n\geqslant 3$.
\end{abstract}

\section{Introduction}

It is well known that vanishing Krein parameters of association schemes have important consequences in combinatorics, for example in determining the feasibility of parameter sets for distance-regular graphs, and in placing bounds upon the orders of generalised polygons. Vanishing Krein parameters may also be used to say something about subsets of the vertices of an association scheme. The key result of this paper is a simple observation (Theorem \ref{MyKreinTheorem}) which states that given certain vanishing Krein parameters, either the possible sums of eigenspaces containing a Delsarte design are constrained, or else the Delsarte design consists of exactly half of the vertices of the association scheme. This leads to various interesting consequences which we explore in this paper. We will assume the reader is familiar with the basic theory of association schemes, but we refer to \cite[Chapter 30]{Lint:2001aa} for background and notation.

The starting point for this work is the following theorem (Theorem \ref{thm:KreinProjection}) of Cameron, Goethals and Seidel \cite{CameronGoethalsSeidel1978_Krein}. A short proof of it was given by Martin \cite{Martin2001}. It states that the Schur product of two vectors\footnote{This is the entrywise product of two vectors.}, each lying in precisely one eigenspace, projects trivially to a third eigenspace if the corresponding Krein parameter vanishes. In particular, this gives us combinatorial meaning, since if the two vectors in question are characteristic vectors of subsets, then their Schur product indicates the intersection of the two sets. 

\begin{theorem}[{\cite[Proposition 5.1]{CameronGoethalsSeidel1978_Krein}}]\label{thm:KreinProjection}
Let $(\Omega, \mathcal{R})$ be a $d$-class association scheme and let $h, i, j \in \{1,\ldots, d\}$.
Write the simultaneous eigenspaces for $(\Omega, \mathcal{R})$ as $V_\ell$.
 If $u \in V_i$ and $v \in V_j$, and we have a vanishing Krein parameter $q_{i j}^h = 0$, then the Schur product $u \circ v$ has trivial projection
 to $V_h$.
\end{theorem}

Cameron, Goethals, and Seidel \cite{CameronGoethalsSeidel1978_Strongly} proved that strongly regular graphs with either $q_{11}^1=0$ or $q_{22}^2=0$ have strongly regular subconstituents about every vertex.
Moreover, they characterised such graphs, when connected, as being a pentagon, a Smith graph  (cf. \cite{Smith1975}), or the complement of a Smith graph. Similar results have been investigated for distance-regular graphs with slightly larger diameter
\cite{Godsil1992,JurisicKoolen2002}. There is also an important relation between vanishing Krein parameters and the \emph{triple intersection numbers}.
Given $x,y,z \in V(\Gamma)$, we define the triple intersection numbers
\[
p_{r,s,t}^{x,y,z} = | \{u : d(x,u) =r, d(y, u)=s, d(z,u)=t \} |.
\]

\begin{theorem}[\cite{CoolsaetJurisic2008}]\label{KreinAndTripleIntersection}
Let $(\Omega, \mathcal{R})$ be a $d$-class association scheme, and let $Q$ denote the matrix of dual eigenvalues of the association scheme. Then
$q_{ij}^h = 0$ if and only if
\[
\sum_{r,s,t = 0}^d Q_{ri} Q_{sj} Q_{th} p_{r,s,t}^{x,y,z} = 0,
\]
for all $x,y,z \in \Omega$.
\end{theorem}

Theorem \ref{KreinAndTripleIntersection} has been used to show that distance-regular graphs with certain intersection arrays do not exist \cite{CoolsaetJurisic2008, JurisicVidali2012, Urlep2012, Vidali2018}.

We now reach the main theorem of this paper. Theorem \ref{MyKreinTheorem} shows, that given certain vanishing Krein parameters, Delsarte designs must 
be constrained to a smaller subset of the eigenspaces in which they lie, or else consist of exactly half of the vertices.

\begin{theorem}\label{MyKreinTheorem}
Let $(\Omega, \mathcal{R})$ be an association scheme, and let $\theta$ be a subset of the vertices, such that its characteristic vector $\chi_\theta$
satisfies
\[
\chi_\theta \in V_0 \perp \left(\bigperp_{\ell \in S} V_\ell\right),
\]
where $S \subseteq \{1, \ldots, d \}$.
If there is some $h \in S$, such that $q_{i,j}^h=0$ for all $i, j \in S$, then 
\[
\chi_\theta \in V_0 \perp \left(\bigperp_{\ell \in S \backslash \{h\}} V_\ell\right),
\]
or $|\theta|=\frac{1}{2}|\Omega|$.
\end{theorem}

The proof of this result will appear in Section \ref{section:proof}. 

\begin{xmpl}
Consider the Johnson scheme $J(8,4)$ on the $4$-subsets $\Omega$ of $\{1,\ldots,8\}$. 
We will order the simultaneous eigenspaces in the natural cometric ordering.
We have $q_{ij}^1=0$ for all $i,j\in\{1,4\}$. Suppose $\theta$ is a subset of $\Omega$ such that 
$\chi_\theta\in V_0\perp V_1\perp V_4$. Then by Theorem \ref{MyKreinTheorem}, 
(i) $\chi_\theta\in V_0\perp V_4$, or (ii) 
$\theta$ consists of half the elements of $\Omega$.
It is not difficult to undertake a complete enumeration of these examples. It turns out that $\theta$ is one of four examples up to symmetry:
\begin{enumerate}[(a)]
\item A $3-(8,4,1)$ design equivalent to the design arising from the points and planes of  the affine space $\mathsf{AG}(3,2)$. In this case $\chi_\theta\in V_0\perp V_4$ and $|\theta| = 14$.
\item The unique flag-transitive $3-(8,4,3)$ design.  In this case $\chi_\theta\in V_0\perp V_4$ and $|\theta| = 42$.
\item All of the 4-subsets containing one element. In this case $\chi_\theta\in V_0\perp V_1$ and $|\theta| = 35 = \frac{|\Omega|}{2}$.
\item A set\footnote{Consider $\PSL(2,5)$ acting transitively on 6 elements (i.e., the natural action on the projective line $\PG(1,5)$). Then $\PSL(2,5)$ has two orbits $\mathcal{O}$ and $\mathcal{O}'$ of size 10 on 3-subsets.
Adjoin $7$ to each $S$ in $\mathcal{O}$
and likewise, adjoin 8 to every element of $\mathcal{O}'$ to produce a total of twenty 4-subsets $\Lambda$ of $\{1,\ldots,8\}$. Now, take the set $\Delta$ of every 4-subset of $\{1,2,3,4,5,6\}$ containing $\{1,2\}$. Then $|\Delta\cup \Lambda|=15+20=35$ and has characteristic vector lying in $V_0\perp V_1\perp V_4$.}
 of half the elements of $\Omega$, with stabiliser $S_5$ in $S_8$. In this case $\chi_\theta\in V_0\perp V_1 \perp V_4$ and $|\theta| = 35 = \frac{|\Omega|}{2}$.
\end{enumerate}
\end{xmpl}

For $T \subseteq \{1, \ldots, d\}$, a \emph{(Delsarte) $T$-design} of $(\Omega, \mathcal{R})$ is a subset $\theta \in \Omega$ such that $\chi_\theta E_i = 0$ for all $i \in T$. Reformulated in the language of $T$-designs, Theorem \ref{MyKreinTheorem} states that for $S = \{1, \ldots, d\} \backslash T$ if $q_{i,j}^h$ vanishes for some $h \in S$ and all $i, j \in S$ then either $|\theta| = |\Omega|/2$ or $\theta$ is a $(T \cup \{h\})$-design.

In order to use Theorem \ref{MyKreinTheorem} we require at a minimum that $q_{hh}^h=0$. 
Moreover, the smallest such setting is where $S=\{h\}$. 
A subset $\theta \in \Omega$ is called an \emph{intriguing set of type $h$} (cf. \cite{BambergMetsch2019}) if $\chi_\theta \in V_0 \perp V_h$.
Hence an intriguing set of type $i$ is a $\{1, \ldots, d \} \backslash \{h\}$-design.
Intriguing sets are often known by other names in different settings. For example, \emph{Cameron-Liebler line classes} (cf. \cite{CameronLiebler1982,DeBoeckRodgersStormeSvob2019,DeBoeckStormeSvob2016,Metsch2017, RodgersStormeVansweevelt2018}) and \emph{Boolean degree 1 functions} (cf. \cite{FilmusIhringer2019}) are all variations of intriguing sets. 
We have the following special case of Theorem \ref{MyKreinTheorem} for intriguing sets of type $h$.

\begin{corollary}\label{cor:KreinIntriguing}
If $q_{hh}^h= 0$, then a nontrivial intriguing set of type $h$ contains exactly half of the vertex set.
\end{corollary}

Corollary \ref{cor:KreinIntriguing} yields alternative proofs of three results in the literature: (i) a nontrivial $m$-ovoid of a generalised
quadrangle of order $(s,s^2)$ is a hemisystem  \cite{CameronGoethalsSeidel1978_Strongly} (see also Corollary \ref{CameronGoethalsSeidel1978_Strongly});
(ii) a nontrivial $m$-ovoid of a dual polar space of the form $\DQ(2d, q)$, $\DW(2d-1,q)$, $\DH(2d-1,q^2)$ (for $d \ge 3$) is a hemisystem
\cite[Theorem 1.1]{Bamberg:2018aa} (see also Theorem \ref{DQDWDH}); (iii)
 a nontrivial relative $m$-ovoid of a generalised quadrangle of order $(q, q^2)$, containing a doubly subtended quadrangle of order $(q,q)$, 
 is a relative hemisystem \cite{BambergLee2017} (see also Theorem \ref{bamberglee}).

\begin{xmpl}
The Clebsch graph, the Schl\"afli graph, and the complement of the Higman-Sims graph, are all strongly regular graphs, with parameters $(16,10, 6,6)$, $(27, 16,10, 8)$, and $(100, 77, 60, 56)$, respectively. In each of these graphs, $q_{11}^1=0$, hence by Corollary \ref{cor:KreinIntriguing}, intriguing sets of type $1$ in these graphs consist of half of the vertices.
\end{xmpl}

\section{Some background and notational conventions}

We remind the reader that we can calculate Krein parameters of an association scheme $(\Omega, \mathcal{R})$ from the matrix of eigenvalues $P$, or
the matrix of dual eigenvalues $Q$.

\begin{lemma}[cf. {\cite[Theorem 2.3.2]{Brouwer:1989aa}}]\label{calculating_krein}
In an association scheme $(\Omega, \mathcal{R})$, let
the $k_\ell$ be the valencies, and let the $m_h$ be the multiplicities (of the simultaneous eigenspaces). Then
\[
q_{ij}^h=\frac{1}{|\Omega|m_h}\sum_{\ell=0}^d k_\ell Q_{\ell i}Q_{\ell j}Q_{\ell h}=\frac{m_im_j}{|\Omega|}\sum_{\ell=0}^d\frac{P_{i\ell}P_{j\ell}P_{h\ell}}{k_\ell^2}.
\]
\end{lemma}

There is a bijection between the power-set of the vertices and the set of all $\{0,1\}$-vectors by taking the characteristic vector of a set.  We write $\chi_S$ for the characteristic vector of a subset $S$ of a domain that is clear from the context. It is clear that Schur multiplication of two $\{0,1\}$-vectors produces another $\{0,1\}$-vector with a one in an entry if and only if the corresponding entry is one in both of the original vectors. Thus intersection of two vertex subsets is described by Schur multiplication of their characteristic vectors. That is, $\chi_{S_1 \cap S_2} = \chi_{S_1} \circ \chi_{S_2}$
for all subsets $S_1, S_2$ of our domain. Moreover, the all-ones vector $\allones$ is the identity under Schur multiplication, and $\allones - \chi_S = \chi_{S^c}$, where $S^c$ is the complement of $S$ in our domain. 

We will need one more item of notation. Recall that the eigenspaces of the association scheme in Theorem \ref{MyKreinTheorem} are denoted $V_i$.
Each $V_i$ has a linear projection map $E_i$; a \emph{minimal idempotent} for the association scheme. Recall that the $E_i$ have the following property:
$E_i E_j=\delta_{i,j} E_i$. So different minimal idempotents are orthogonal.

A \emph{partial linear space} is a (nonempty) incidence geometry of points and lines such that any two distinct points are contained in at most one line, 
and every line contains at least two points. An \emph{$m$-ovoid} of a partial linear space is a subset $S$ of the points such that every line is incident with
$m$ elements of $S$.
If every line has $s+1$ points incident with it, then an $\frac{s+1}{2}$-ovoid is a \emph{hemisystem}. The \emph{trivial} $m$-ovoids are the empty set and the whole point set, that is, $m$ is $0$ or $s+1$.

\section{Proof of the main result}\label{section:proof}

\begin{proof}[Proof of Theorem \ref{MyKreinTheorem}]
Let $v = \chi_\theta$. Then $v \in V_0 \perp \left(\bigperp_{\ell \in S} V_\ell\right)$, and hence $v = \lambda \allones + \sum_{\ell \in S} v_\ell$, where $v_\ell \in V_\ell$ and $v_\ell E_h = 0$ for all $\ell \neq h$. 
If $v_h = 0$, then $ vE_h  = 0$ and so
\[
\chi_\theta \in V_0 \perp \left(\bigperp_{\ell \in S \backslash \{h\}} V_\ell\right).
\]
Instead, take $v_h \neq 0$.
Now, 
\[
v \circ v = \lambda^2 \allones + \sum_{i \in S} 2\lambda v_i + \sum_{i,j \in S}v_i \circ v_j,
\]
since $\allones$ is the identity under Schur multiplication. Thus $(v \circ v)E_h = 2\lambda v_h$, since $ v_\ell E_h = 0$ for all $\ell \neq h$, and $(v_i \circ v_j) E_h = 0$ for all $i, j \in S$, by Theorem \ref{thm:KreinProjection} (as $q_{ij}^h = 0$ for all $i,j\in S$). However, note that $v \circ v = v = \lambda \allones + \sum_{\ell \in S} v_\ell$, since $v $ is a $\{0,1\}$-vector. Thus $  (v \circ v)E_h = v_h$.
Equating these expressions yields $2\lambda v_h = v_h$, and since $v_h$ is nontrivial, $\lambda = \frac{1}{2}$.
Therefore, 
\[
|\theta| = \allones \cdot v =\lambda(\allones \cdot \allones) + \sum_{\ell \in S} (\allones \cdot v_\ell)=\lambda(\allones \cdot \allones) = \tfrac{1}{2}|\Omega|.\qedhere
\]
\end{proof}

\section{Vanishing Krein parameters in cometric, $Q$-bipartite and $Q$-antipodal schemes}

A scheme is \emph{cometric} (or \emph{Q-polynomial}) with respect to the ordering $\{E_i\}_{i=0}^d$, if
\begin{enumerate}[noitemsep,topsep=0pt]
	\item $q_{ij}^h = 0$ if $i + j < h$ or $0 \le h < |i-j|$,
	\item $q_{ij}^{i+j} \neq 0$ for all $i$ and $j$ such that $i + j \le d$.
\end{enumerate}

Thus such schemes have a large proportion of vanishing Krein parameters. Indeed, Terwilliger \cite{Terwilliger1992} was interested in cometric schemes for this reason.
A scheme is called \emph{metric} if the equivalent conditions hold when $q_{ij}^h$ is replaced by $p_{ij}^h$.
 Many of the classical association schemes are both metric and cometric. Schemes which are cometric but not metric are relatively rare, hence the interest in the \emph{Penttila-Williford scheme}. We consider this scheme further in Section \ref{PenttilaWillifordScheme}. Moreover, most association schemes arising from geometries are cometric, such as the dual polar spaces, for example.

\begin{corollary}\label{cor:cometric}
Let $(\Omega, \mathcal{R})$ be a cometric association scheme, and let $\theta$ be a subset of the vertices, such that
\[
\chi_\theta \in V_0 \perp \left(\bigperp_{i \in S} V_i \right)\perp V_h,
\]
where $S \subseteq \{1, \ldots, \lceil \frac{h}{2} \rceil -1 \}$ for some $h$.
If $q_{hh}^h=0$ then 
\[
\chi_\theta \in V_0 \perp \left(\bigperp _{i \in S} V_i\right),
\]
or $|\theta|=\frac{1}{2}|\Omega|$.
\end{corollary}

An association scheme is \emph{primitive} if all of its associate graphs, $\Gamma_i$, are connected. Otherwise it is called \emph{imprimitive}. 
The following conjecture, due to Bannai and Ito, if true, would imply that, for sufficiently large diameter, most of the Krein parameters of a primitive distance-regular graph vanish.

\begin{conjecture}[Bannai and Ito {\cite[p. 312]{BannaiIto1984}}]
For sufficiently large diameter, a primitive association scheme is metric if and only if it is cometric.
\end{conjecture}

There is no guarantee that $q_{ii}^i=0$ for some $i$, even in a cometric scheme.
However, we can say a bit more, particularly in the case of imprimitive cometric association schemes.
For a cometric association scheme, define $a_i^* = q_{1i}^i$, $b_i^* = q_{1, i+1}^i$, and $c_i^* = q_{1, i-1}^i$. The \emph{Krein array}, or \emph{dual intersection array}, is
\[
\{ b_0^*, b_1^*, \ldots, b_{d-1}^*; c_1^*, c_2^*, \ldots, c_d^*\}.
\]
Furthermore, $a_i^* + b_i^* + c_i^* = q_{ii}^0$.

A scheme is \emph{bipartite} if $p_{ij}^h = 0$ when $i+j+h$ is odd, and
antipodal if $b_j = c_{d-j}$ for all $j$ except possibly $j = \lfloor \frac{d}{2} \rfloor$.
Dually, a scheme is \emph{$Q$-bipartite} if $q_{ij}^h = 0$ when $i+j+h$ is odd, and $Q$-antipodal if $b_j^* = c_{d-j}^*$ for all $j$ except possibly $j = \lfloor \frac{d}{2} \rfloor$. Note that $a_i^*=0$ for all $i$ implies that a scheme is $Q$-bipartite.

\begin{xmpl}[{\cite[5.5.10]{Dam:2016aa}}]
For $d \ge 2$, a distance-regular graph is called \emph{almost $Q$-bipartite} if $a_i^*=0$ for $i < d$ and $a_d^*>0$. An almost $Q$-bipartite distance-regular graph is either the halved $(2d+1)$-cube, the folded $(2d+1)$-cube, or the collinearity graph of the dual polar space $\DH(2d-1, q^2)$. Since $a_1^*=q_{11}^1=0$, it follows from Corollary \ref{cor:KreinIntriguing} that an intriguing set of type $1$ consists of half the vertex set.
\end{xmpl}

An imprimitive metric association scheme other than a cycle is bipartite, or antipodal, or both {\cite[Theorem 4.2.1]{Brouwer:1989aa}}.
Suzuki \cite{Suzuki1998imprimitive} proved that an imprimitive cometric association scheme is $Q$-bipartite, or $Q$-antipodal, or both $Q$-bipartite and $Q$-antipodal, or has either four or six classes. Cerzo and Suzuki \cite{CerzoSuzuki2009} later showed that the exceptional four-class case does not exist, followed by a similar non-existence result by Tanaka and Tanaka for the exceptional six-class case \cite{TanakaTanaka2011}. Therefore, an imprimitive cometric association scheme other than a cycle is $Q$-bipartite, or $Q$-antipodal, or both.

The $Q$-bipartite condition is very strong, leading to the following result for constraining Delsarte designs.

\begin{theorem} \label{thm:QbipartiteKrein}
Let $(\Omega, \mathcal{R})$ be a $Q$-bipartite association scheme, and let $\theta$ be a nontrivial subset of the vertices, such that
\[
\chi_\theta \in V_0 \mathop{\perp}_{i \in S} V_i,
\]
where $S \subseteq \{1, \ldots, d \}$, such that $i$ is odd for all $i \in S$.
Then $|\theta|=\frac{1}{2}|\Omega|$.
\end{theorem}

\begin{proof}
Since $i, j, h$ are odd for all $i, j, h \in S$, so too is $i+j+h$. Hence $q_{ij}^h=0$ for all $i, j, h \in S$, as $(\Omega, \mathcal{R})$ is $Q$-bipartite. Assume for a contradiction that $|\theta|\neq\frac{1}{2}|\Omega|$.
If we fix some $h \in S$, then 
\[
\chi_\theta \in V_0 \perp \left(\bigperp_{\ell \in S\backslash \{h\}} V_\ell\right),
\] by Theorem \ref{MyKreinTheorem}. Repeating this argument eventually constrains $\chi_\theta$ to lie in $V_0 = \langle \allones \rangle$, but this is a contradiction, since $\theta$ is a nontrivial proper subset of the vertices. Thus $|\theta|=\frac{1}{2}|\Omega|$.
\end{proof}

\begin{xmpl}
The \emph{Taylor graphs} are distance-regular graphs with intersection array 
\[
\{k, \mu, 1; 1 , \mu, k\}.
\]
 They are cometric with respect to two orderings, and are $Q$-bipartite (cf. \cite[p. 431]{Brouwer:1989aa}). Hence by Theorem \ref{thm:QbipartiteKrein}, intriguing sets of type $1$ and $3$ consist of half of the vertex set, as do $\{2\}$-designs.
\end{xmpl}

\section{Partial geometries and generalised quadrangles}

Partial geometries are a class of partial linear spaces having strongly regular collinearity graphs.
To be more precise, a \textit{partial geometry} $\mathsf{pg}(s,t,\alpha)$ is a partial linear space such that
every line contains $s+1$ points, every point is incident with $t+1$ lines and for a point
$P$ and line $\ell$ which are not incident, there are $\alpha$ points on $\ell$ collinear
with $P$. A partial geometry gives rise to a strongly regular graph with parameters
\[((s+1)(st+\alpha)/\alpha,s(t+1),s-1+t(\alpha-1),\alpha(t+1))\]
simply by considering the vertices to be the points, and adjacency to be collinearity of points.
By default, the association scheme arising from the partial geometry will be that arising from
the strongly regular collinearity graph, and the minimal idempotents will have the natural ordering\footnote{by  
their corresponding multiplicities $1$, $\frac{st (s+1)  (t+1)}{\alpha (s + t+1-\alpha)}$, $\frac{s(s+1-\alpha)  (s t+\alpha)}{\alpha (s + t+1-\alpha)}$.}.
A partial geometry is \emph{thick} if $s,t>1$. Bruck nets and generalised quadrangles occur naturally as partial geometries.
Indeed, a $\mathsf{pg}(s,t,\alpha)$ is a generalised quadrangle precisely when $\alpha=1$.

\begin{theorem}\label{thm:partal_geom_krein}
Let $\Gamma$ be a thick $\pg(s,t,\alpha)$. Then the Krein parameter $q_{ii}^i$ is zero if and only if $i=2$, 
\[
\alpha=\frac{s^2-t-1+\sqrt{t (t+1-s^2)}}{s - 1},
\]
 and $t\ge s^2$. Moreover, in the positive case, $\alpha=1$ if and only if $t=s^2$.
\end{theorem}

\begin{proof}
Note that $1\le \alpha \le s+1,t+1$.
There are $(s+1)(st+\alpha)/\alpha$ points, and the matrix of eigenvalues for the collinearity graph (which is strongly regular) is:
\[
P=
\begin{bmatrix}
1&s(t+1)&\tfrac{st}{\alpha}(s+1-\alpha)\\
1&s-\alpha&\alpha-s-1\\
1&-t-1&t
\end{bmatrix}.
\]
Let us calculate the $q_{ii}^i$ (using Lemma \ref{calculating_krein}):

\begin{align*}
q_{11}^1&=1+\frac{(s-\alpha)^3}{(s(t+1))^2}+\frac{(\alpha-s-1)^3}{(\tfrac{st}{\alpha}(s+1-\alpha))^2},\\
q_{22}^2&=1+\frac{(-t-1)^3}{(s(t+1))^2}+\frac{t^3}{\left(\tfrac{st}{\alpha}(s+1-\alpha)\right)^2}.
\end{align*}

\begin{description}
\item[Case $q_{11}^1=0$.]
If $\alpha=s+1$, then
\[
q_{11}^1=1-\frac{1}{s^2(t+1)^2}
\]
which is clearly nonzero. So assume $\alpha<s+1$. Then we can cross-multiply and take the numerator:
\[
(s(t+1))^2\left(\tfrac{st}{\alpha}(s+1-\alpha)\right)^2+(s-\alpha)^3\left(\tfrac{st}{\alpha}(s+1-\alpha)\right)^2+(\alpha-s-1)^3(s(t+1))^2=0.
\]
Since $s,t,\alpha >0$,
\[
0=(s^2(t+1)^2+(s-\alpha)^3)t^2+(\alpha-s-1)(t+1)^2\alpha^2
\]
and hence
\[
(s^2(t+1)^2+(s-\alpha)^3)t^2=(s+1-\alpha)(t+1)^2\alpha^2.
\]
Now $0\le (s-\alpha)^3$ and so
\[
s^2(t+1)^2t^2\le (s+1-\alpha)(t+1)^2\alpha^2.
\]
Hence, 
\[
s^2t^2\le (s+1-\alpha)\alpha^2\le (s+1-\alpha)(t+1)^2
\]
since $\alpha \le t+1$. So
\[
s^2 \le (s+1-\alpha)(\tfrac{t+1}{t})^2 < s(\tfrac{t+1}{t})^2
\]
as $\alpha \ge 1$. Thus
\[
s < (\tfrac{t+1}{t})^2=\left(1+\frac{1}{t}\right)^2\le \left(\frac{3}{2}\right)^2=2+\frac{1}{4}.
\]
Since $s\ge 2$, we have $s=2$. When $t>2$, we have $(1+\tfrac{1}{t})^2<2$; a contradiction. Therefore, $t=2$ and $\alpha=1$ because $0< \alpha < s+1$. We can now substitute in our values for $s, t, \alpha$ to compute $q_{11}^1=65/72 \neq 0$.

\item[Case $q_{22}^2=0$.]

If $\alpha=s+1$, then the second relation is empty, so we assume $\alpha<s+1$. As before, we cross-multiply and take the numerator:
\[
(s(t+1))^2\left(\tfrac{st}{\alpha}(s+1-\alpha)\right)^2-(t+1)^3\left(\tfrac{st}{\alpha}(s+1-\alpha)\right)^2+t^3(s(t+1))^2=0.
\]
Since $s,t,\alpha >0$,
\[
0=s^2(s+1-\alpha)^2-(t+1)(s+1-\alpha)^2+t\alpha^2=(s^2-t-1)(s+1-\alpha)^2+t\alpha^2
\]
and hence
\[
\alpha = \frac{s^2-t-1\pm \sqrt{t (t+1-s^2)}}{s - 1}.
\]
Now for $\alpha$ to be a real number, it is necessary that $t+1\ge s^2$. If $t+1=s^2$, then $\alpha=0$ (a contradiction). Therefore, $t\ge s^2$.
Furthermore, $t \ge s^2$ implies that $\sqrt{t (t+1-s^2)} \ge \sqrt{t} \ge s \ge 1$ and $s^2 - t -1 \le -1$. Now, $\alpha > 0$ and so 
\begin{equation}\label{alpha}
\alpha = \frac{s^2-t-1 + \sqrt{t (t+1-s^2)}}{s - 1}.
\end{equation}

Moreover, from Equation \eqref{alpha}, $\alpha=1$ is equivalent to  $ t (t+1-s^2) = (s^2-t- s)^2$, which after rearrangement, yields
$t( s^2 - 2s + 1) = s^2(s^2 -2s +1)$; that is, $t=s^2$.\qedhere
\end{description}
\end{proof}

By Corollary \ref{cor:KreinIntriguing}, Theorem \ref{thm:partal_geom_krein} yields a result on hemisystems of partial geometries.

\begin{corollary}
Let $\Gamma$ be a $\pg\left(s,t,\frac{s^2-t-1+\sqrt{t (t+1-s^2)}}{s - 1}\right)$, where $s,t>1$. Then any nontrivial $m$-ovoid of $\Gamma$ is a hemisystem.
\end{corollary}

As a direct consequence of this result, we re-obtain a theorem by Cameron, Goethals and Seidel \cite{CameronGoethalsSeidel1978_Strongly}.

\begin{corollary} \label{CameronGoethalsSeidel1978_Strongly}
Let $\Gamma$ be a $\mathsf{GQ}(s,s^2)$, where $s>1$. Then any nontrivial $m$-ovoid of $\Gamma$ is a hemisystem.
\end{corollary}

For $\alpha = 2$, the only putative partial geometries with vanishing Krein parameters are $\pg(4, 27, 2)$, $\pg(5, 32, 2)$, and $\pg(7, 54, 2)$ \cite{Makhnev2002}. Makhnev \cite{Makhnev2002} showed that there is no partial geometry $\pg(5, 32, 2)$ while \"{O}sterg\aa rd and Soicher \cite{OstergardSoicher2018} showed that there is no partial geometry $\pg(4, 27, 2)$, leaving the only open case for $\alpha=2$ being $\pg(7, 54, 2)$.

The partial geometry $\pg(6, 80, 3)$ induces $\pg(5, 32, 2)$ in the neighbourhood of a vertex, and so there is no partial geometry  $\pg(6, 80, 3)$ \cite{Makhnev2002}. The smallest open case for a partial geometry with $\alpha=3$ and a vanishing Krein parameter is thus $\pg(7, 75, 3)$. For $\alpha \ge 4$, the existence problem of partial geometries with $q_{22}^2=0$ appears to be completely open.

\begin{problem}
Does there exist a $\pg(7, 54, 2)$ or $\pg(7, 75, 3)$?
\end{problem}

\begin{problem}
For which values of $\alpha \ge 4$ do partial geometries exist with $q_{22}^2=0$?
\end{problem}

\section{Dual polar spaces}

We present here an alternative proof of the main result of \cite{Bamberg:2018aa}. 

\begin{theorem}[{\cite[Theorem 1.1]{Bamberg:2018aa}}]\label{DQDWDH}
The only nontrivial $m$-ovoids that exist in $\DQ(2d, q)$, $\DW(2d-1,q)$ and $\DH(2d-1,q^2)$, for $d \ge 3$, are hemisystems
(i.e., $m = (q+1)/2$).
\end{theorem}

\begin{proof}
The matrix of eigenvalues for $\DH(5, q^2)$, $\DQ(6,q)$, and $\DW(5,q)$ are as follows (see \cite [Theorem 4.3.6]{VanhoveThesis}):
\[
P_{\DH(5,q^2)}=
\begin{pmatrix}
1 & q(q^4+q^2+1) & q^4 (q^4 + q^2 +1) & q^9 \\
1 & q^3 +  q -1 & q(q^3-q^2-1) & -q^4 \\
1 & -q^2 +q -1 & -q(q^2-q+1) & q^3 \\
1 & -q^4 - q^2 -1 & q^2(q^4 + q^2 +1) & -q^6
\end{pmatrix},
\]
and
\[
P_{\DQ(6,q)}=P_{\DW(5,q)}=
\begin{pmatrix}
1 & q(q^2+q+1) & q^3(q^2+q+1) & q^6\\
1 & q^2+q-1 & q(q^2-q-1) & -q^3\\
1 & -1 & -q^2 & q^2\\
1 & -q^2-q-1 & q(q^2+q+1) & -q^3
\end{pmatrix}.
\]

Computing the Krein parameters with Lemma \ref{calculating_krein}, $q_{11}^1>0$, $q_{22}^2 >0$, and $q_{33}^3=0$ in each instance.
Now $m$-ovoids of $\DH(5, q^2)$, $\DQ(6,q)$, and $\DW(5,q)$ are intriguing sets of type $3$, and so by 
Corollary \ref{cor:KreinIntriguing} the only nontrivial $m$-ovoids of $\DH(5,q^2)$, $\DQ(6,q)$, and $\DW(5,q)$ are hemisystems.
The result for general dimension $d>3$ follows by noting that by taking the generators on a fixed point, and then projecting,
maps an $m$-ovoid to an $m$-ovoid.
\end{proof}

It is also worth noting here that nothing more can be said about designs in $\DW(5,q)$ than what can be said about $\DQ(6,q)$ from a purely algebraic perspective, since the intersection numbers are the same for both.

\begin{remark}
The conditions of Theorem \ref{MyKreinTheorem} being satisfied does not imply the existence of a $T$-design, $\theta$, such that $|\theta| = \frac{1}{2}|\Omega|$. 
Such a $T$-design may or may not exist. By way of example, the dual polar space $\DW(5,3)$ has Krein parameter $q_{33}^3=0$, 
which implies that every $m$-ovoid of this space is a hemisystem. However, it has no hemisystems. While the dual polar space $\DQ(6,3)$ also has Krein parameter $q_{33}^3=0$, and does have hemisystems (see \cite{Bamberg:2018aa}). 
\end{remark}

The following is a long-standing open problem.
\begin{problem}
When do $\DQ(2d, q)$, $\DW(2d-1,q)$ and $\DH(2d-1,q^2)$ have hemisystems, for $d \ge 3$?
\end{problem}

\section{The Penttila-Williford scheme} \label{PenttilaWillifordScheme}

Let $\mathcal{Q}$ be a generalised quadrangle of order $(s,t)$ and let $\mathcal{Q}'$ be a generalised subquadrangle of $\mathcal{Q}$ of order $(s,t')$. 
Let $P$ be a point of $\mathcal{Q}$ not in $\mathcal{Q}'$. Then $P^\perp\cap\mathcal{Q}'$ is an
ovoid $\mathcal{O}_P$ of $\mathcal{Q}'$ \emph{subtended} by $P$. If there is exactly one other point $P'$ such that $\mathcal{O}_P$ is also subtended by 
$P'$, then we say $\mathcal{O}_P$ is \emph{doubly subtended}. If every ovoid of $\mathcal{Q}'$ is doubly subtended, then we say that 
$\mathcal{Q}'$ is doubly subtended in $\mathcal{Q}$.

Let $\mathcal{Q}$ be a generalised quadrangle of order $(s,s^2)$, where $s >2$, and let $\mathcal{Q}'$ be a doubly subtended generalised quadrangle of order $(s,s)$ contained in $\mathcal{Q}$. For example, we could take the classical generalised quadrangle $\Q^-(5,q)$ and the subquadrangle
$\Q(4,q)$. Let $\Omega$ be the set of points of $\mathcal{Q}\setminus \mathcal{Q}'$, and define the following relations on $\Omega\times \Omega$:
\begin{itemize}
\item $R_0$ is the identity relation;
\item $(X,Y)\in R_1$ if and only if $X$ and $Y$ are not collinear in $\mathcal{Q}$ and $|\mathcal{O}_X\cap \mathcal{O}_Y|=1$;
\item $(X,Y)\in R_2$ if and only if $X$ and $Y$ are not collinear in $\mathcal{Q}$ and $|\mathcal{O}_X\cap \mathcal{O}_Y|=s+1$;
\item $(X,Y)\in R_3$ if and only if $X$ and $Y$ are collinear in $\mathcal{Q}$;
\item $(X,Y)\in R_4$ if and only if $\mathcal{O}_X= \mathcal{O}_Y$.
\end{itemize}
Then the relations $\{R_i\}_{i=0}^4$ form a primitive association scheme (see \cite{Penttila:2011aa}),
which we call the \emph{Penttila-Williford scheme}. The Penttila-Williford scheme is $Q$-bipartite (see \cite{Penttila:2011aa}), and so we have the following
immediate consequence of Corollary \ref{thm:QbipartiteKrein}.

\begin{theorem} \label{lem:PenttilaWillifordIntriguing}
Let $\theta$ be a nontrivial intriguing set of type $1$, intriguing set of type $3$, or a $\{2,4\}$-design of the Penttila-Williford association scheme. Then $|\theta|=\frac{1}{2}|\Omega|$.
\end{theorem}

Bamberg and Metsch \cite{BambergMetsch2019} recently described the intriguing sets of the Penttila-Williford scheme arising from $\Q^-(5,q) \backslash \Q(4, q)$, $q>2$. They show that an intriguing set of type $1$ or type $3$ must consist of exactly half of the vertices of the scheme. Theorem \ref{lem:PenttilaWillifordIntriguing} uses only algebraic arguments and hence generalises the result of Bamberg and Metsch.

The following result of Bamberg and Lee follows as a corollary of Theorem \ref{lem:PenttilaWillifordIntriguing} and the fact that a relative $m$-cover is an intriguing set of type $1$  \cite[Corollary 3.4]{BambergLee2017}.

\begin{theorem}[\cite{BambergLee2017}]\label{bamberglee}
Given a generalised quadrangle of order $(q, q^2)$ containing a doubly subtended quadrangle of order $(q,q)$, a nontrivial relative $m$-ovoid is a relative hemisystem.
If it exists, $q$ is even.
\end{theorem}

We now look at a generalisation of the Penttila-Williford scheme which, although a natural extension, does not appear to be in the literature. We state it explicitly and consider some of its properties by drawing on tools from Cossidente and Pavese who studied the generalisation of relative $m$-ovoids \cite{Cossidente:2019wi}.
 Let $\Omega$ be the set of points of $\Q^-(2n-1,q)$ not contained in a fixed nondegenerate hyperplane $\Pi$.
 So $|\Omega|=q^{n-1}(q^{n-1}-1)$. 
We define $\sigma\colon \Omega\to \Omega$ as follows. Let $P$ be a point of $\Omega$.
 Then the line joining $\Pi^\perp$ and $P$ is a hyperbolic line and so contains a unique second point $P'$ of $\Omega$.
Let $\sigma$ be the central collineation of $\PG(2n-1,q)$ having axis $\Pi$ and
centre $\Pi^\perp$, mapping $P$ to $P'$. Then $\sigma$ commutes with the polarity $\perp$ and so stabilises $\Omega$.
Moreover, since $PP'$ has only two points of $\Q^-(2n-1,q)$ on it, we have $\sigma(P')=P$ and hence $\sigma^2$ is the identity.
It turns out that $\sigma$ is independent of the choice of $P$. We write $X\sim Y$ as a shorthand notation for the `collinear and not equal'
relation.

\begin{theorem}\label{genPW}
Let $q$ be a prime power, greater than $2$, and let $\Omega$ be the set of points of $\Q^-(2n-1,q)$ not contained in a fixed nondegenerate hyperplane $\Pi$.
Define the following relations on $\Omega$:
\begin{align*}
R_0&=\{(X,Y)\in \Omega\times \Omega\mid X=Y\},
\\
R_1&=\{(X,Y)\in \Omega\times \Omega\mid Y\not\sim X\sim Y^\sigma\},
\\
R_2&=\{(X,Y)\in \Omega\times \Omega\mid Y\not\sim X\not\sim Y^\sigma\},
\\
R_3&=\{(X,Y)\in \Omega\times \Omega\mid Y\sim X\not\sim Y^\sigma\},
\\
R_4&=\{(X,Y)\in \Omega\times \Omega\mid X=Y^\sigma\}.
\end{align*}
Then $(\Omega,\{R_0,R_1,R_2,R_3,R_4\})$ is an association scheme with matrix of eigenvalues:
\[
\begin{bmatrix}
1 & ( q^{n-2}-1)(q^{n-1}+1)  & q^{n-2} ( q-2) (q^{n-1}+1)  &  ( q^{n-2}-1)(q^{n-1}+1)  & 1\\
1 & {q}^{n-1}+1 & 0 & -\left( {q}^{n-1}+1\right)  & -1\\
1 & q^{n-2}-1 & -2q^{n-2} & q^{n-2}-1 & 1\\
1 & -\left( q^{n-2}-1\right)  & 0 & q^{n-2}-1 & -1\\
1 &-q^{n - 2}(q - 2)-1& 2q^{n-2}( q-2)  & -q^{n - 2}(q - 2)-1& 1
\end{bmatrix}.
\]
\end{theorem}

\begin{proof}
We will sketch the proof since it is similar to the proof of \cite[Theorem 1]{Penttila:2011aa}.
First, one verifies that the relations $R_i$ are symmetric and the intersection numbers $p_{ij}^h$ are well-defined. Next, 
the parameters 
$p_{ij}^h$ can be computed geometrically and we obtain the intersection matrices $L_1$, $L_2$, $L_3$, and $L_4$, which we list in Appendix \ref{ap:GenPentWillIntersection}.
The eigenvectors for $L_1$ give the matrix of eigenvalues as described above.
\end{proof}

\begin{lemma}
The association scheme described in Theorem \ref{genPW} is not metric. It is cometric precisely when $n=3$, in which case there is a single cometric ordering.
\end{lemma}
\begin{proof}
The dual intersection matrices $L_i^*$ have been computed using Lemma \ref{calculating_krein} and are given in Appendix \ref{ap:GenPentWill}. Observe that $L_1^*$ is tridiagonal when $n=3$, and hence the association scheme is cometric with respect to the given ordering of eigenspaces. When $n>3$, there are $10$, $6$, $10$, and $6$ entries of $L_1^*$, $L_2^*$, $L_3^*$, and $L_4^*$, respectively, which are non-zero and not on the diagonal. Any ordering of the eigenspaces induces a permutation $\sigma \in S_4$ on the indices. Now, since $(L_i^*)_{jh}=q_{ij}^h$, it follows that $(L_{\sigma(i)})_{jh} = q_{\sigma(i) \sigma(j)}^{\sigma(h)}$, and so $L_{\sigma(i)} = T^\top L_i T$, where $T$ is the permutation matrix induced by $\sigma$. Note that the diagonal of $L_i$ is fixed by $T$, and so the rows and columns of $L_1^*$ and $L_3^*$ clearly cannot be permuted to a tridiagonal form, as they have too many non-diagonal, non-zero entries. Similarly, $L_2^*$ and $L_3^*$ have too few entries; even if they were permuted into a tridiagonal form, it would result in $b^*_x$ and $c^*_y$ being zero for some $x, y \in \{1, 2, 3, 4\}$, which is not possible. Hence the association scheme is not cometric for $n >3$. Since the number of non-diagonal, non-zero entries of $L_2^*$, $L_3^*$, and $L_4^*$ are constant for $n\ge 3$, the same argument shows that only the given ordering is a cometric ordering for $n=3$. A similar treatment shows that none of the intersection matrices (see Appendix \ref{ap:GenPentWillIntersection}) can be expressed in tridiagonal form, for $n \ge 3$, and so the association scheme is not metric.
\end{proof}

\begin{theorem} \label{lem:PenttilaWillifordIntriguing2}
Let $\theta$ be a nontrivial intriguing set of type $1$, intriguing set of type $3$, or a $\{2,4\}$-design of the association scheme described in Theorem \ref{genPW}. Then $|\theta|=\frac{|\Omega|}{2}$. If $n=3$ and $\theta$ is a nontrivial $\{2,3\}$-design, then either $\theta$ is a $\{1,2,3\}$-design or $|\theta|=\frac{|\Omega|}{2}$.
\end{theorem}

\begin{proof}
Recall that $q_{ij}^h = (L_i^*)_{jh}$. By inspecting the dual intersection matrices (Appendix \ref{ap:GenPentWill}), we see that the following Krein parameters always vanish:
\[
q_{11}^1,\enspace q_{33}^3,\enspace q_{13}^1,\enspace q_{31}^1,\enspace q_{33}^1,\enspace q_{11}^3,\enspace q_{13}^3,\enspace q_{31}^3.
\]
Additionally, when $n=3$, the following also vanish:
\[
q_{44}^1, \enspace q_{1,4}^1, \enspace q_{41}^1.\qedhere
\]
\end{proof}

A \emph{relative $m$-ovoid} of $\Q^-(2n - 1, q)$, with respect to a nondegenerate hyperplane $\Pi$, is a subset
$\mathcal{R}$ of points of $\Q^-(2n - 1, q) \setminus \Pi$ such that every generator of $\Q^-(2n - 1, q)$ not contained in $\Pi$ meets 
$\mathcal{R}$ in $m$ points.
A relative $m$-ovoid is said to be nontrivial or proper if it is nonempty and not the entire set of points of $\Q^-(2n - 1, q) \setminus \Pi$. 
It is a relative hemisystem if it comprises half of the points of $\Q^-(2n - 1, q) \setminus \Pi$. Now the points of $\Q^-(2n - 1, q) \setminus \Pi$
in a generator of $\Q^-(2n - 1, q)$ (not contained in $\Pi$) form a clique of size $q^{n-2}$ for the $R_3$-relation and so has inner distribution vector
$a=(1,0,0,q^{n-2}-1,0)$. Now 
\[
aQ=\left(q^{n-2},0,\frac{q^{n-2}(q-2) (q+1)  (q^{n-1}-1)}{2 (q-1)},\frac{1}{2} q^{n-1} (q^{n-1}-1),\frac{q^{n-1} (q^{n-2}-1)}{q-1}\right)
\]
and so a relative $m$-ovoid is an intriguing set of type 1 (by \cite[3.3, 3.4]{roos} and \cite[Proposition 2.5.2]{Brouwer:1989aa}).

\begin{corollary}[{\cite[Theorem 2.4]{Cossidente:2019wi}}]
Let $\mathcal{R}$ be a proper relative $m$-ovoid of $\Q^-(2n - 1, q)$. Then $q$ is even and $\mathcal{R}$ is a relative hemisystem.
\end{corollary}

Relative hemisystems of $\Q^-(4n+1, q)$ are known to exist for $q$ even and $n\ge 2$ \cite[Theorem 3.3]{Cossidente:2019wi}. Other than for $\Q^-(7,2)$, where no relative hemisystem exists \cite[Remark 3.5]{Cossidente:2019wi}, the question of their existence is open for $\Q^-(4n-1, q)$, $q$ even and $n \ge 2$.

\begin{problem}
Do there exist relative hemisystems of $\Q^-(4n-1, q)$, $q$ even and $n \ge 2$?
\end{problem}

\section{Generalised octagons}

A finite generalised octagon is a partial linear space such that (i) there are no ordinary $n$-gons in the geometry with $n<8$, and (ii) there exists an ordinary octagon in the geometry. If there are parameters $(s,t)$ such that every line is incident with $s+1$ points and every point is incident with $t+1$ lines, then we say that the generalised octagon has order $(s,t)$. If a generalised octagon of order $(s,t)$ has $s,t\ge 3$, then we say it is \emph{thick}. The dual incidence structure  of a generalised octagon of order $(s,t)$ (whereby points and lines are interchanged) is again a generalised octagon, but with order $(t,s)$. The only known examples of finite thick generalised octagons are the Ree-Tits octagons of order $(q,q^2)$ and their duals, where $q$ is an odd power of 2. Indeed, they are the natural geometries for the exceptional groups of Lie type $\phantom{}^2F_4(q)$. Now it is customary in the theory of generalised polygons with \emph{gonality} at least 6 to use the term \emph{distance-$j$-ovoid} for a maximum\footnote{in particular, obtaining a natural upper bound} set of points mutually at distance $j$, and reserve the term \emph{ovoid} for a maximum set of pairwise opposite points. However, for convenient notation, we will still stipulate that an $m$-ovoid of a generalised polygon is a set of points such that every line meets it in $m$ points; in the original vein of \cite[\S3]{Thas:1989us}. The authors believe the following
result is new.

\begin{theorem}\label{octagon}
A generalised octagon of order $(s, s^2)$ does not contain a nontrivial $m$-ovoid.
\end{theorem}

\begin{proof}
Let $\mathcal{O}$ be a generalised octagon of order $(s,t)$. The collinearity graph on the points of $\mathcal{O}$ is distance-regular with intersection array
\[
\{ s(t+1), st, st, st ; 1, 1, 1, t+1 \},
\]
and hence with intersection matrix
\[
L=\begin{bmatrix}
0 & s(t+1) & 0 & 0 & 0 \\
1 & s-1 & st &  0 & 0 \\
0 & 1 & s-1 & st & 0 \\
0 & 0 & 1 & s-1 & st \\
0 & 0 & 0 & t+1 & (s-1)(t+1)
\end{bmatrix}.
\]
The eigenvalues of $L$ are $s-1$ , $s-1 - \sqrt{2st}$, $s -1 +  \sqrt{2st}$, $-(t+1)$,  and $s (1 + t)$, from which we can compute the eigenvectors of $L$ via standard sequences and the eigenvalue multiplicities via Biggs' formula (cf. \cite[pp. 13--4]{Dam:2016aa}). After scaling by the eigenvalue multiplicities, the eigenvectors form the columns of the matrix of dual eigenvalues, $Q$. With the aid of Mathematica \cite{Mathematica}, we computed the matrix of eigenvalues to be
\[Q=
\begin{bmatrix}
 1 & 1 & 1 & 1 & 1 \\
 1 & \frac{s-1}{s (t+1)} & -\frac{1}{s} & \frac{\sqrt{2st}+s-1}{s (t+1)} & \frac{-\sqrt{2st}+s-1}{s(t+1)} \\
 1 & -\frac{1}{s t} & \frac{1}{s^2} & \frac{(t-1)\sqrt{s}+\sqrt{2t} (s-1)}{s^{3/2} t (t+1)} & \frac{(t-1)\sqrt{s}-\sqrt{2t} (s-1)}{s^{3/2} t (t+1)} \\
 1 & \frac{1-s}{s^2 t(t+1)} & -\frac{1}{s^3} & \frac{-\sqrt{2s/t}+s-1}{s^2 t(t+1)} & \frac{\sqrt{2s/t}+s-1}{s^2t(t+1)} \\
 1 & \frac{1}{s^2 t^2} & \frac{1}{s^4} & -\frac{1}{s^2 t^2} & -\frac{1}{s^2 t^2} \\
\end{bmatrix}\cdot \,\mathrm{diag}(1,m_1,m_2,m_3,m_4)
\]
where
\begin{align*}
m_1 &=\frac{st (s+1) (t+1) (s t+1) \left(s^2 t^2+1\right)}{4 \left((s-1)^2 t+\sqrt{2st} (s-1) (t-1) + s (t-1)^2+2 s t\right)},\\
m_2 &= \frac{st (s+1) (t+1) (s t+1) \left(s^2 t^2+1\right)}{4 \left((s-1)^2 t-\sqrt{2st} (s-1) (t-1) +s (t-1)^2+2 s t\right)},\\
m_3 &= \frac{st (s+1) (t+1) \left(s^2 t^2+1\right)}{2 (s+t)},\\
m_4 &= \frac{s^4 (s t+1) \left(s^2 t^2+1\right)}{(s+t) \left(s^2+t^2\right)}.
\end{align*}

Consider the point-set of a line. It has size $s+1$ and is a clique with respect to the first relation. Hence it has inner distribution vector $a = (1, s, 0, 0, 0)$.
Considering the MacWilliams transform (see \cite[Proposition 2.5.2]{Brouwer:1989aa}),
\begin{align*}
aQ = \bigg(&1+s,\; \frac{s(1+s)t(s+t)(1+st)(1+s^2 t^2)}{4(t+s^2t +\sqrt{2st} - \sqrt{2st^3} + s(t-1)(t-1+\sqrt{2st}))},\; 0,\\
& \frac{s(1+s)t(s+\sqrt{2st}+t)(1+s^2 t^2)}{2(s+t)},\; \frac{s^4(s-\sqrt{2st}+t)(1+st)(1+s^2 t^2)}{(1+t)(s+t)(s^2+t^2)} \bigg),
\end{align*}
we see that only $aQ_2 =0$ for $s,t>0$. 
Hence by \cite[(3.27)]{delsarte} a line is a $\{2\}$-design. Thus an $m$-ovoid of $\mathcal{O}$ is a $\{1,3,4\}$-design (by \cite[Corollary 3.3]{roos}). In order to apply Theorem \ref{MyKreinTheorem} to $m$-ovoids of $\mathcal{O}$, we thus require $q_{22}^2=0$.

Since we are interested in when $q_{22}^2$ vanishes, by Lemma \ref{calculating_krein}, we need only compute:
\[
\sum_{\ell=0}^4 k_\ell Q_{\ell2}^3 = \frac{(s-1)(1+s)^4(s^2-t)t^3(1+t)^3(1+st)^3(s^4+t^2)(1+s^2t ^2)^3}{64s^5 \bigg( t+ s^2 t - \sqrt{2st}+ \sqrt{2st^3} - s(t-1)(1-t+\sqrt{2st})\bigg)^3}.
\]
The numerator of this sum, and hence $q_{22}^2$, is clearly zero when $s^2=t$ or $s=1$.
Thus by Theorem \ref{MyKreinTheorem} a non-trivial $m$-ovoid is a hemisystem  when $s^2 = t$. However, $2st=2s^3$ is a perfect square by Feit-Higman \cite{FeitHigman}, so $2$ divides $s$ resulting in an odd number of points on a line. A  hemisystem must contain exactly half of the points on every line and so an $m$-ovoid must be trivial.
\end{proof}

Recalling that the Ree-Tits octagon has order $(q, q^2)$:

\begin{corollary}
The Ree-Tits octagon does not contain any nontrivial $m$-ovoids.
\end{corollary}

\begin{remark}
The dual Ree-Tits octagon of order $(4,2)$ has $m$-ovoids for every possible $m\in\{1,2,3,4\}$. They are not difficult to construct by computer if one uses
the orbits of a Sylow $13$-subgroup of $\phantom{}^2F_4(2)'$.
\end{remark}

\begin{problem}
 When do $m$-ovoids of the dual Ree-Tits octagons exist? Are there infinite families?
\end{problem}

\appendix
\section{Intersection matrices for the generalisation of the Penttila-Williford association scheme} \label{ap:GenPentWillIntersection}

{\small
\[
L_1 = 
\begin{bmatrix}
0 & (q^{n-2}-1)(q^{n-1}+1) & 0 & 0 & 0\\
1 & q^{2n-4} & q^{n + 1} - 2 q^n & q^{n-2} (q^{n-2}-q+1)-2 & 0\\
0 & q^{n-2}(q^{n-2}-1) &  (q^{n-2}-1)((q-2) q^{n-2}+1) & q^{n-2}(q^{n-2}-1) & 0\\
0 & q^{n-2} (q^{n-2}-q+1)-2 &q^{n + 1} - 2 q^n & q^{2n-4} & 1\\
0 & 0 & 0 & (q^{n-2}-1)(q^{n-1}+1)  & 0
\end{bmatrix}
\]

\[
L_2=
\begin{bmatrix}
 0 & 0 & (q-2) q^{n-2} \left(q^{n-1}+1\right) & 0 & 0 \\
 0 & (q-2) q^{2 (n-2)} & (q-2) q^{n-2} \left((q-2) q^{n-2}+1\right) & (q-2) q^{2 (n-2)} & 0 \\
 1 & (q^{n-2}-1) ((q-2) q^{n-2}+1) & q^{n-2} \left((q-2)^2 q^{n-2}+3 q-8\right) & (q^{n-2}-1) ((q-2) q^{n-2}+1) & 1 \\
 0 & (q-2) q^{2 (n-2)} & (q-2) q^{n-2} \left((q-2) q^{n-2}+1\right) & (q-2) q^{2 (n-2)} & 0 \\
 0 & 0 & (q-2) q^{n-2} \left(q^{n-1}+1\right) & 0 & 0 \\
\end{bmatrix}
\]

\[
L_3=
\begin{bmatrix}
 0 & 0 & 0 & (q^{n-1}+1) (q^{n-2}-1) & 0 \\
 0 & q^{n-4} \left(q^n-q^3+q^2\right)-2 & (q-2) q^{2 (n-2)} & q^{2 (n-2)} & 1 \\
 0 & q^{n-2} \left(q^{n-2}-1\right) & (q^{n-2}-1) ((q-2) q^{n-2}+1) & q^{n-2} \left(q^{n-2}-1\right) & 0 \\
 1 & q^{2 (n-2)} & (q-2) q^{2 (n-2)} & q^{n-4} \left(q^n-q^3+q^2\right)-2 & 0 \\
 0 & (q^{n-1}+1) (q^{n-2}-1)& 0 & 0 & 0 \\
\end{bmatrix}
\]

\[
L_4=
\begin{bmatrix}
 0 & 0 & 0 & 0 & 1 \\
 0 & 0 & 0 & 1 & 0 \\
 0 & 0 & 1 & 0 & 0 \\
 0 & 1 & 0 & 0 & 0 \\
 1 & 0 & 0 & 0 & 0 \\
\end{bmatrix}
\]
}

\newpage
\section{Dual intersection matrices for the generalisation of the Penttila-Williford association scheme} \label{ap:GenPentWill}

{\small
\[
L_1^* =
\begin{bmatrix}
 0 & \frac{\left(q^n-q\right) \left(q^{n-2}-1\right)}{2 (q+1)} & 0 & 0 & 0 \\
 1 & 0 & \frac{(q-2) \left(q^{2 (n-1)}-1\right)}{2  \left(q^2-1\right)} & 0 & \frac{\left(q^n+q\right) \left(q^{n-2}-q\right)}{2 \left(q^2-1\right)} \\
 0 & \frac{\left(q^n-q\right) \left(q^{n-2}-1\right)}{2 (q+1)^2} & 0 & \frac{\left(q^n-q\right) \left(q^{n-1}-q\right)}{2  (q+1)^2} & 0 \\
 0 & 0 & \frac{(q-2) \left(q^n-q\right) \left(q^{n-2}-1\right)}{2  \left(q^2-1\right)} & 0 & \frac{\left(q^n-q\right) (q^{n-2}-1)}{2 (q^2-1)} \\
 0 & \frac{\left(q^n-q\right) \left(q^{n-2}-q\right)}{2 (q+1)^2} & 0 & \frac{q\left(q^{n-1}-1\right)^2}{2 (q+1)^2} & 0 \\
\end{bmatrix}
\]

\[
L_2^*=
\begin{bmatrix}
 0 & 0 & \frac{(q-2) \left(q^{2 n-2}-1\right)}{2 (q-1) } & 0 & 0 \\
 0 & \frac{(q-2) \left(q^{2 n-2}-1\right)}{2 \left(q^2-1\right)} & 0 & \frac{(q-2) \left(q^{2 n-1}-q\right)}{2  \left(q^2-1\right)} & 0 \\
 1 & 0 & \frac{(q-2)^2 q^{2 n-2}+(q-3) q^{n}-2 q^2+7q-4 }{2 (q-1)^2 } & 0 & \frac{\left(q^n-q^2\right) \left((q-2) q^{n-2}+1\right)}{2 (q-1)^2 } \\
 0 & \frac{(q-2) \left(q^n-q\right) \left(q^{n-2}-1\right)}{2 \left(q^2-1\right)} & 0 & \frac{(q-2) \left(q^{n-1}-1\right) \left(q^n+2 q+1\right)}{2 (q^2-1)} & 0 \\
 0 & 0 & \frac{(q-2) \left(q^n-q\right) \left((q-2) q^{n-2}+1\right)}{2 (q-1)^2 } & 0 & \frac{(q-2) \left(q^{n-1}-1\right)^2}{2 (q-1)^2 } \\
\end{bmatrix}
\]

\[
L_3^*=
\begin{bmatrix}
 0 & 0 & 0 & \frac{q^{2 n-1}-q}{2  (q+1)} & 0 \\
 0 & 0 & \frac{(q-2) \left(q^{2 n-1}-q\right)}{2  \left(q^2-1\right)} & 0 & \frac{q^{2 n-1}-q}{2 (q^2-1)} \\
 0 & \frac{\left(q^n-q\right) \left(q^{n-1}-q\right)}{2 (q+1)^2} & 0 & \frac{\left(q^n-q\right) \left(q^n+2 q+1\right)}{2 (q+1)^2} & 0 \\
 1 & 0 & \frac{(q-2) \left(q^{n-1}-1\right) \left(q^n+2 q+1\right)}{2 \left(q^2-1\right)} & 0 & \frac{\left(q^n+q+2\right) \left(q^{n-1}-q\right)}{2 \left(q^2-1\right)} \\
 0 & \frac{\left(q^{n-1}-1\right)^2}{2 (q+1)^2} & 0 & \frac{\left(q^n-q\right) \left(q^n+q+2\right)}{2 (q+1)^2} & 0 \\
\end{bmatrix}
\]

\[
L_4^*=
\begin{bmatrix}
 0 & 0 & 0 & 0 & \frac{\left(q^n+q\right) \left(q^{n-2}-1\right)}{2 (q-1)} \\
 0 & \frac{\left(q^n+q\right) \left(q^{n-2}-q\right)}{2 \left(q^2-1\right)} & 0 & \frac{q^{2 n-1}-q}{2 (q^2-1)} & 0 \\
 0 & 0 & \frac{(q^{n-2}-1) \left(q^{n+1}-2 q^n+q^2\right)}{2 (q-1)^2 } & 0 & \frac{(q^n-q) (q^{n-2}-1)}{2 (q-1)^2} \\
 0 & \frac{\left(q^n-q\right) (q^{n-2}-1)}{2 (q^2-1)} & 0 & \frac{\left(q^n+q+2\right) \left(q^{n-1}-q\right)}{2 \left(q^2-1\right)} & 0 \\
 1 & 0 & \frac{(q-2) \left(q^{n-1}-1\right)^2}{2 (q-1)^2} & 0 & \frac{q^{2 n-2}-(q^2-4q+5) q^{n-1}+(4-3 q) q}{2 (q-1)^2 } \\
\end{bmatrix}
\]
}


\begin{thebibliography}{10}

\bibitem{Bamberg:2018aa}
J.~Bamberg, J.~Lansdown, and M.~Lee.
\newblock On {$m$}-ovoids of regular near polygons.
\newblock {\em Des. Codes Cryptogr.}, 86(5):997--1006, 2018.

\bibitem{BambergLee2017}
J.~Bamberg and M.~Lee.
\newblock A relative {$m$}-cover of a {H}ermitian surface is a relative
  hemisystem.
\newblock {\em J. Algebraic Combin.}, 45(4):1217--1228, 2017.

\bibitem{BambergMetsch2019}
J.~Bamberg and K.~Metsch.
\newblock On intriguing sets of the {P}enttila-{W}illiford association scheme.
\newblock {\em Linear Algebra Appl.}, 582:327--345, 2019.

\bibitem{BannaiIto1984}
E.~Bannai and T.~Ito.
\newblock {\em Algebraic combinatorics. {I}}.
\newblock The Benjamin/Cummings Publishing Co., Inc., Menlo Park, CA, 1984.

\bibitem{Brouwer:1989aa}
A.~E. Brouwer, A.~M. Cohen, and A.~Neumaier.
\newblock {\em Distance-regular graphs}, volume~18 of {\em Ergebnisse der
  Mathematik und ihrer Grenzgebiete (3) [Results in Mathematics and Related
  Areas (3)]}.
\newblock Springer-Verlag, Berlin, 1989.

\bibitem{CameronGoethalsSeidel1978_Krein}
P.~J. Cameron, J.-M. Goethals, and J.~J. Seidel.
\newblock The {K}rein condition, spherical designs, {N}orton algebras and
  permutation groups.
\newblock {\em Nederl. Akad. Wetensch. Indag. Math.}, 40(2):196--206, 1978.

\bibitem{CameronGoethalsSeidel1978_Strongly}
P.~J. Cameron, J.-M. Goethals, and J.~J. Seidel.
\newblock Strongly regular graphs having strongly regular subconstituents.
\newblock {\em J. Algebra}, 55(2):257--280, 1978.

\bibitem{CameronLiebler1982}
P.~J. Cameron and R.~A. Liebler.
\newblock Tactical decompositions and orbits of projective groups.
\newblock {\em Linear Algebra Appl.}, 46:91--102, 1982.

\bibitem{CerzoSuzuki2009}
D.~R. Cerzo and H.~Suzuki.
\newblock Non-existence of imprimitive {$Q$}-polynomial schemes of exceptional
  type with {$d=4$}.
\newblock {\em European J. Combin.}, 30(3):674--681, 2009.

\bibitem{CoolsaetJurisic2008}
K.~Coolsaet and A.~Juri\v{s}i\'{c}.
\newblock Using equality in the {K}rein conditions to prove nonexistence of
  certain distance-regular graphs.
\newblock {\em J. Combin. Theory Ser. A}, 115(6):1086--1095, 2008.

\bibitem{Cossidente:2019wi}
A.~Cossidente and F.~Pavese.
\newblock Relative {$m$}-ovoids of elliptic quadrics.
\newblock {\em Discrete Math.}, 342(5):1481--1488, 2019.

\bibitem{DeBoeckRodgersStormeSvob2019}
M.~De~Boeck, M.~Rodgers, L.~Storme, and A.~\v{S}vob.
\newblock Cameron-{L}iebler sets of generators in finite classical polar
  spaces.
\newblock {\em J. Combin. Theory Ser. A}, 167:340--388, 2019.

\bibitem{DeBoeckStormeSvob2016}
M.~De~Boeck, L.~Storme, and A.~\v{S}vob.
\newblock The {C}ameron-{L}iebler problem for sets.
\newblock {\em Discrete Math.}, 339(2):470--474, 2016.

\bibitem{delsarte}
P.~Delsarte.
\newblock An algebraic approach to the association schemes of coding theory.
\newblock {\em Philips Res. Rep. Suppl.}, (10):vi+97, 1973.

\bibitem{FeitHigman}
W.~Feit and G.~Higman.
\newblock The nonexistence of certain generalized polygons.
\newblock {\em J. Algebra}, 1:114--131, 1964.

\bibitem{FilmusIhringer2019}
Y.~Filmus and F.~Ihringer.
\newblock Boolean degree 1 functions on some classical association schemes.
\newblock {\em J. Combin. Theory Ser. A}, 162:241--270, 2019.

\bibitem{Godsil1992}
C.~D. Godsil.
\newblock Krein covers of complete graphs.
\newblock {\em Australas. J. Combin.}, 6:245--255, 1992.

\bibitem{JurisicKoolen2002}
A.~Juri\v{s}i\'{c} and J.~Koolen.
\newblock Krein parameters and antipodal tight graphs with diameter 3 and 4.
\newblock {\em Discrete Math.}, 244(1-3):181--202, 2002.
\newblock Algebraic and topological methods in graph theory (Lake Bled, 1999).

\bibitem{JurisicVidali2012}
A.~Juri\v{s}i\'{c} and J.~Vidali.
\newblock Extremal 1-codes in distance-regular graphs of diameter 3.
\newblock {\em Des. Codes Cryptogr.}, 65(1-2):29--47, 2012.

\bibitem{Makhnev2002}
A.~A. Makhnev.
\newblock On the nonexistence of strongly regular graphs with the parameters
  {$(486,165,36,66)$}.
\newblock {\em Ukra\"{\i}n. Mat. Zh.}, 54(7):941--949, 2002.

\bibitem{Martin2001}
W.~J. Martin.
\newblock Symmetric designs, sets with two intersection numbers and {K}rein
  parameters of incidence graphs.
\newblock {\em J. Combin. Math. Combin. Comput.}, 38:185--196, 2001.

\bibitem{Metsch2017}
K.~Metsch.
\newblock A gap result for {C}ameron-{L}iebler {$k$}-classes.
\newblock {\em Discrete Math.}, 340(6):1311--1318, 2017.

\bibitem{OstergardSoicher2018}
P.~R.~J. \"{O}sterg\aa rd and L.~H. Soicher.
\newblock There is no {M}c{L}aughlin geometry.
\newblock {\em J. Combin. Theory Ser. A}, 155:27--41, 2018.

\bibitem{Penttila:2011aa}
T.~Penttila and J.~Williford.
\newblock New families of {$Q$}-polynomial association schemes.
\newblock {\em J. Combin. Theory Ser. A}, 118(2):502--509, 2011.

\bibitem{RodgersStormeVansweevelt2018}
M.~Rodgers, L.~Storme, and A.~Vansweevelt.
\newblock Cameron-{L}iebler {$k$}-classes in {${\rm PG}(2k+1,q)$}.
\newblock {\em Combinatorica}, 38(3):739--757, 2018.

\bibitem{roos}
C.~Roos.
\newblock On antidesigns and designs in an association scheme.
\newblock {\em Delft Progr. Rep.}, 2(2):98--109, 1982.

\bibitem{Smith1975}
M.~S. Smith.
\newblock On rank {$3$} permutation groups.
\newblock {\em J. Algebra}, 33:22--42, 1975.

\bibitem{Suzuki1998imprimitive}
H.~Suzuki.
\newblock Imprimitive {$Q$}-polynomial association schemes.
\newblock {\em J. Algebraic Combin.}, 7(2):165--180, 1998.

\bibitem{TanakaTanaka2011}
H.~Tanaka and R.~Tanaka.
\newblock Nonexistence of exceptional imprimitive {$Q$}-polynomial association
  schemes with six classes.
\newblock {\em European J. Combin.}, 32(2):155--161, 2011.

\bibitem{Terwilliger1992}
P.~Terwilliger.
\newblock The subconstituent algebra of an association scheme. {I}.
\newblock {\em J. Algebraic Combin.}, 1(4):363--388, 1992.

\bibitem{Thas:1989us}
J.~A. Thas.
\newblock Interesting pointsets in generalized quadrangles and partial
  geometries.
\newblock {\em Linear Algebra Appl.}, 114/115:103--131, 1989.

\bibitem{Urlep2012}
M.~Urlep.
\newblock Triple intersection numbers of {$Q$}-polynomial distance-regular
  graphs.
\newblock {\em European J. Combin.}, 33(6):1246--1252, 2012.

\bibitem{Dam:2016aa}
E.~R. van Dam, J.~H. Koolen, and H.~Tanaka.
\newblock Distance-regular graphs.
\newblock {\em Electron. J. Combin.}, \#DS22, 2016.

\bibitem{Lint:2001aa}
J.~H. van Lint and R.~M. Wilson.
\newblock {\em A course in combinatorics}.
\newblock Cambridge University Press, Cambridge, second edition, 2001.

\bibitem{VanhoveThesis}
F.~Vanhove.
\newblock {\em Incidence geometry from an algebraic graph theory point of
  view}.
\newblock PhD thesis, Ghent University, 2011.

\bibitem{Vidali2018}
J.~Vidali.
\newblock Using symbolic computation to prove nonexistence of distance-regular
  graphs.
\newblock {\em Electron. J. Combin.}, 25(4):Paper 4.21, 10, 2018.

\bibitem{Mathematica}
{Wolfram Research Inc.}
\newblock Mathematica, {V}ersion 12.0.
\newblock Champaign, IL, 2021.

\end{thebibliography}

\end{document}